\documentclass[reqno,a4paper,12pt]{amsart}
\usepackage{dsfont}
\usepackage{amsmath,amssymb,amsthm}
\usepackage{amsfonts}
\usepackage{tikz}
\usetikzlibrary{arrows}
\usepackage[T1]{fontenc}
\usepackage{fourier}
\usepackage{fullpage}
\usepackage{mathtools}
\usepackage{url}
\allowdisplaybreaks

\newtheorem{theorem}{Theorem}[section]
\newtheorem{corollary}[theorem]{Corollary}
\newtheorem{lemma}[theorem]{Lemma}

\newtheorem{question}[theorem]{Question}
\theoremstyle{remark}
\newtheorem{remark}[theorem]{Remark}

\theoremstyle{definition}
\newtheorem{definition}[theorem]{Definition}

\title{Planarity and genus of sparse random bipartite graphs}
\author{Tuan Anh Do$^{*}$, Joshua Erde$^{*}$, and Mihyun Kang$^{*}$}
\address{Institute of Discrete Mathematics,\\ 
Graz University of Technology, \\
Steyrergasse 30,\\
8010 Graz,
Austria.}
\email{\{do,erde,kang\}@math.tugraz.at}
\thanks{Supported by Austrian Science Fund (FWF): I3747, W1230}


\begin{document}

\begin{abstract}
The genus of the binomial random graph $G(n,p)$ is well understood for a wide range of $p=p(n)$. Recently, the study of the genus of the random bipartite graph $G(n_1,n_2,p)$, with partition classes of size $n_1$ and $n_2$, was initiated by Jing and Mohar, who showed that when $n_1$ and $n_2$ are comparable in size and $p=p(n_1,n_2)$ is significantly larger than $(n_1n_2)^{-\frac{1}{2}}$, the genus of the random bipartite graph has a similar behaviour to that of the binomial random graph. 

In this paper we show that there is a threshold for planarity of the random bipartite graph at $p=(n_1n_2)^{-\frac{1}{2}}$ and investigate the genus close to this threshold, extending the results of Jing and Mohar. It turns out that there is qualitatively different behaviour in the case where $n_1$ and $n_2$ are comparable, when whp the genus is linear in the number of edges, than in the case where $n_1$ is asymptotically smaller than $n_2$, when whp the genus behaves like the genus of a sparse random graph $G(n_1,q)$ for an appropriately chosen $q=q(p,n_1,n_2)$.
\end{abstract}

\maketitle

\section{Introduction}
\subsection{Motivation}

The binomial random graph model $G(n,p)$, introduced by Gilbert \cite{Gilbert}, is a random variable distributed on the subgraphs of the complete graph $K_n$, whose distribution is given by including each edge independently with probability $p$. This model, together with the closely related Erd\H{o}s-R\'{e}nyi random graph model $G(n,m)$, has been extensively studied since its introduction (see \cite{Frieze, Bollobas,E-R,Janson}). One particularly striking feature of this model is the `concentration' that it displays for many graph parameters, that is, the value of certain graph parameters in the model $G(n,p)$ are with high probability ({\em whp} for short) determined, at least asymptotically, as a function solely of $n$ and $p$.

The particular graph parameter that we will focus on in this paper is the genus.
The \emph{genus} of a graph $G$ is the smallest genus of an orientable surface on which $G$ can be embedded, in other words, the smallest $g\in \mathbb{N}$ such that $G$ can be embedded on a sphere with $g$ handles attached (see \cite{GT01} for more background on topological graph theory).
We will write $g(G)$ for the genus of $G$.

The genus is a key topological property of a graph, which has applications to the design of graph algorithms (e.g., colouring problems \cite{Mohar2} and the manufacture of electrical circuits \cite{Gibb, Lip}). In addition, recently, results on the genus of random bipartite graphs \cite{Mohar} were used to give a polynomial-time approximation scheme for the genus of dense graphs \cite{Jing}.

The genus of the binomial random graph $G(n,p)$ was first studied by Archdeacon and Grable \cite{Arch}, who showed that for large enough $p$, more precisely if $p^2(1-p^2) \geq \frac{8 \log^4 n}{n}$, whp the genus of $G(n,p)$ is $(1+o(1))\frac{1}{12}pn^2$. Since the number of edges in $G(n,p)$ is whp $(1+o(1))\frac{pn^2}{2}$ when $p = \omega(n^{-2})$, it follows that in the range of $p$ considered by Archdeacon and Grable, whp $g(G(n,p) = (1+o(1)) \frac{1}{6}e(G(n,p))$. R\"{o}dl and Thomas \cite{Rodl} extended these results to show that whp the genus of $G(n,p)$ is $(1+o(1))\frac{i}{2(i+2)}\frac{pn^2}{2}$ whenever\footnote{Here and throughout the paper we will use the notation $f(n) \ll g(n)$ to denote that $\lim_{n \rightarrow \infty} f(n)/g(n) = 0$.}
\[
n^{-\frac{i}{i+1}} \ll  p \ll n^{-\frac{i-1}{i}}  \quad (i\in \mathbb N),
\]
and so for these ranges of $p$, the genus of $G(n,p)$ is also whp asymptotically linear in the number of edges, with the linear factor increasing from $\frac{1}{6}$ to $\frac{1}{2}$ as $i$ increases. More recently, this was extended to even sparser random graphs by Dowden, Kang, and Krivelevich \cite{Kang}.\footnote{They stated their results in the Er\H{o}s-R\'{e}nyi model, but they worked in the binomial random model.}

\begin{theorem}[\cite{Kang}, Theorems 1.1 and 1.2]\label{(1)and(2)}
\
\begin{enumerate}
    \item If $n^{-1} \ll p \ll n^{-1 +o(1)}$, then whp 
$$
g(G(n,p)) = (1+o(1)) \frac{1}{2} \frac{pn^2}{2}.
$$
\item \label{part2} If $p = \frac{d}{n}$ for constant  $d>1$, then whp 
$$g(G(n,p))= (1+o(1))\mu(d)\frac{pn^2}{2}, $$
	where $$\mu(d)=\frac{1}{2}-\frac{1}{d} +\frac{1}{d^2}\sum_{k=1}^\infty\left(d e^{-d}\right)^k \frac{k^{k-2}}{k!}$$
	is an increasing and continuous function on $(1,\infty)$, with 
	$$\lim_{d \rightarrow 1} \mu(d) = 0\quad \text{and}\quad \lim_{d \rightarrow \infty} \mu(d) = \frac{1}{2}.$$
	\end{enumerate}
\end{theorem}

They also gave some results for the genus in the weakly supercritical regime. Note that it is relatively easy to show that when $p = \frac{1-\epsilon}{n}$ for any fixed positive $\epsilon$, whp $G(n,p)$ is planar, i.e., $g(G(n,p))=0$. 

So, in the binomial random graph model the genus is relatively well understood: the threshold for planarity occurs at $p = \frac{1}{n}$, when $p = \frac{d}{n}$ for $d>1$ the genus is whp linear in the number of edges, where this linear factor increases continuously as a function of $d$ from $0$ to $\frac{1}{2}$, and for larger values of $p$ the genus is also whp linear in the number of edges, where this linear factor decreases from $\frac{1}{2}$ to $\frac{1}{6}$ via a series of phases transitions at $p=n^{\frac{-i}{i+1}}$ for each $i\in \mathbb N$.

In this paper we will be interested in the corresponding question in random bipartite graphs. The \emph{binomial random bipartite graph model} $G(n_1,n_2,p)$ is a random variable distributed on the subgraphs of the complete bipartite graph $K_{n_1,n_2}$, which has partition classes $N_1$ and $N_2$ of size $n_1$ and $n_2$ respectively, whose distribution is given by including each edge between $N_1$ and $N_2$ independently with probability $p$. Recently, Jing and Mohar \cite{Mohar} gave an analogue of R\"{o}dl and Thomas' result in this model.

\begin{theorem}[\cite{Mohar}, Theorem 1.3]\label{theorem1.2}
If there exist a positive constant $c$ and $i \in \mathbb{N}$ such that $\frac{1}{c} \leq \frac{n_1}{n_2} \leq c$ and 
\[
(n_1n_2)^{-\frac{i}{2i+1}} \ll p \ll (n_1n_2)^{-\frac{i-1}{2i-1}},
\]
then whp
\[
g(G(n_1,n_2,p)) = (1+o(1)) \frac{i}{2(i+1)}pn_1n_2.
\]
\end{theorem}
So, again, as long as the two partition classes are comparable in size, if $p$ is much larger than $(n_1n_2)^{-\frac{1}{2}}$, then the genus of $G(n_1,n_2,p)$ is whp asymptotically linear in the number of edges, where the linear factor increases as $p$ decreases, going from $\frac{1}{4}$ (when $p$ is constant) up towards $\frac{1}{2}$. They also gave some results in the case where $n_1 \gg 1$ and $n_2$ is a constant (see \cite{Mohar,JM20}).

It is natural to ask if similar behaviour occurs for even smaller $p$, in particular, as in $G(n,p)$, if the genus is linear in the number of edges just above the threshold for planarity. As we will see, the answer to this question depends on whether $n_1$ and $n_2$ are comparable in size or not.

\subsection{Main results} 
Our first result is that whp the binomial random bipartite graph $G(n_1,n_2,p)$ is planar when $p$ is smaller than $\frac{1}{\sqrt{n_1n_2}}$. We note that this result does not depend on the relationship between $n_1$ and $n_2$, and that our later results will imply that this is in fact a sharp threshold for planarity.

\begin{theorem}\label{main1}
Let $G=G(n_1,n_2,p)$ where $1\ll n_1\leq n_2$. If  $p=\frac{d}{\sqrt{n_1n_2}}$ for constant $d<1$, then whp $g(G)=0$, i.e., $G$ is planar.
\end{theorem}

In the supercritical regime, where $p=\frac{d}{\sqrt{n_1n_2}}$ with $d>1$, there is different behaviour according to whether $n_1 = \Theta(n_2)$, in which case we say that the graph is \emph{balanced}, or $n_1 = o(n_2)$, in which case we say that the graph is \emph{unbalanced}.

Firstly in the balanced case we see that the genus is again whp linear in the number of edges, noting that whp the number of edges in $G(n_1,n_2,p)$ is $(1+o(1))pn_1n_2$ for $p \gg (n_1n_2)^{-1}$.
\begin{theorem} \label{main2}
	Let $G=G(n_1,n_2,p)$ where $1 \ll n_1=\lambda n_2$ for constant $\lambda\leq 1$. If $p=\frac{d}{\sqrt{n_1n_2}}$ for constant $d>1$, then whp
    $$g(G)=(1+o(1))\gamma(d,\lambda)pn_1n_2,$$ where 
		$$\gamma(d,\lambda)=\frac{1}{2}-\frac{\lambda+1}{2d\sqrt{\lambda}}+\frac{1}{2d^2}\sum_{k=1}^{\infty}\left(\frac{d}{\sqrt{\lambda}}e^{-\frac{d}{\sqrt{\lambda}}}\right)^k\sum_{\substack{r+s=k,\\ 0\leq r,s\leq k}}\frac{r^{s-1}s^{r-1}}{r!s!}\lambda^re^{-\frac{d(\lambda-1)}{\sqrt{\lambda}}s}.$$
	\end{theorem}
	
As with Theorem \ref{(1)and(2)}, we can also determine the value of the genus for slightly larger values of $p$, giving an intermediary result between Theorems \ref{theorem1.2} and \ref{main2}.
	\begin{theorem}\label{main4}
Let $G=G(n_1,n_2,p)$ where $1 \ll n_1=\lambda n_2$ for constant $\lambda\leq 1$. If $(n_1n_2)^{-\frac{1}{2}} \ll p \ll\left(n_1n_2\right)^{-\frac{1}{2}+o(1)}$, then whp
$$(1-o(1))\frac{1}{2}pn_1n_2 \leq g(G)\leq \frac{1}{2}pn_1n_2.$$
	\end{theorem}
	
	Perhaps surprisingly, in the unbalanced case, the genus behaves very differently in the supercritical regime. Not only will the genus of $G(n_1,n_2,p)$ with $1 \ll n_1 \ll n_2$  whp be sublinear in the number of edges, amazingly it will in fact coincide with the genus of a supercritical random binomial graph $G\left(n_1,\frac{d^2}{n_1}\right)$. As we will shall see later, this is no coincidence.
\begin{theorem}\label{main3}
	Let $G=G(n_1,n_2,p)$ where $1 \ll n_1 \ll n_2$. If $p=\frac{d}{\sqrt{n_1n_2}}$ where $d>1$, then whp 
$$g(G)=(1+o(1))\mu\left(d^2\right)\frac{d^2n_1}{2},$$
where $\mu$ is the function in Theorem \ref{(1)and(2)}.	\end{theorem}

\subsection{Techniques and outline of the paper}
A useful tool for studying the genus of a graph is Euler's formula and, following previous papers on the subject (see \cite{Kang,Mohar}), we will also use this tool for the balanced case.  Given a graph $G$, Euler's formula states that
$$g(G)=\frac{1}{2}\left(e(G)-v(G)-f(G)+\kappa(G)+1\right),$$
where $e(G)$ is the number of edges of $G$, $v(G)$ is the number of vertices of $G$, $f(G)$ is the number of faces of $G$  when embedded on a surface of minimal genus (i.e., a sphere to which $g(G)$ handles have been attached), and $\kappa(G)$ is the number of components of $G$.

In the case where $n_1$ and $n_2$ are comparable in size and $p = \frac{d}{\sqrt{n_1n_2}}$,  writing $G = G(n_1,n_2,p)$, a standard argument tells us that whp $f(G) = o(\sqrt{n_1n_2})$, whereas we can give asymptotic expressions for $e(G)$ and $v(G)$ which are both $\Theta(\sqrt{n_1n_2})$. So, in order to estimate the genus it remains to estimate $\kappa(G)$. Using some results of Johansson \cite[Lemmas 7 and 8]{Johansson} about the emergence of the giant component in $G(n_1,n_2,p)$, we can show that the majority of components in $G(n_1,n_2,p)$ are tree components, and then use first and second moment calculations to estimate the number of such components.

In the case where $n_1 \ll n_2$ however, it becomes difficult to use Euler's formula. Indeed, say again with $p = \frac{d}{\sqrt{n_1n_2}}$, it is not too hard to see that
$$
\mathbb{E}(e(G)) = pn_1n_2 \quad \text{and} \quad v(G) = n_1 + n_2,$$ 
and furthermore, in this case whp the number of isolated vertices in $N_1$ is approximately
\begin{equation}\label{e:isolated}
n_2(1-p)^{n_1} = n_2 - pn_1n_2 + O(n_1),
\end{equation}
and so
$$ \mathbb{E}(\kappa(G)) = \mathbb{E}\left(\kappa_1(G) + \kappa_{\geq 2}(G)\right) = n_2 - pn_1n_2 + O(n_1),
$$
where $\kappa_1$ and $\kappa_{\geq 2}$ are the number of components of order one and at least two respectively in $G$ (note that, since every component of order at least two meets the smaller partition class, the number of components will be asymptotically determined by the number of isolated vertices in the larger partition class), and so the leading order terms all cancel. Moreover, our bound on $f(G)$ is then perhaps large compared to the genus.

In order to get around this we consider an auxilliary graph $H$ on $N_1$, which we call the \emph{$2$-centre} of $G$ (see Section \ref{unbalanced}), formed by joining two vertices $x,y $ with an edge if they are joined by a path $xzy$ of length two in $G$ such that $d(z) =2$. This is similar to a graph considered by Johansson \cite[Section 3.1]{Johansson}, which he called the \emph{even projection}, however we have an extra condition on vertices $z \in N_2$ lying in the path of length two.

By considering the structure of $H$, firstly, we can get a better bound on $f(G)$, showing that $f(G) = o(n_1)$, which in fact holds regardless of the relationship between $n_1$ and $n_2$. However, more importantly we can show that $H$ has a distribution which is very close to that of a binomial random graph. We hope this technique will be useful for answering other questions about the graph $G(n_1,n_2,p)$ when $n_1 \ll n_2$.

	\begin{lemma}\label{lemma5.3} 
	Let $G=G(n_1,n_2,p)$ where $1 \ll n_1 \ll n_2$ and $\left(n_1n_2\right)^{-\frac{1}{2}} \leq p \ll \min \left\{ n_1^{-1}, n_2^{-\frac{1}{2}}\right\}$. For any $\delta > 0$, if $G_1= G(n_1,q_1)$ and $G_2= G(n_1,q_2)$ are binomial random graphs with $q_1=(1-\delta)p^2n_2$ and $ q_2=(1+\delta)p^2n_2$, then 
	$$G_1 \preceq H \preceq G_2,$$
	where $\preceq$ denotes stochastic domination (see Definition \ref{def2.6}).
	\end{lemma}

Furthermore, we can show that whp the genus of $G(n_1,n_2,p)$ is close to the genus of $H$ (see Lemmas \ref{lemma5.1} and \ref{lemma5.2}), and hence we can determine the likely genus of $G(n_1,n_2,p)$ using Theorem \ref{(1)and(2)}.

The rest of the paper is structured as follows. Firstly, in Section \ref{Pre} we provide some notation, definitions, and key facts which we use in the paper. Then, in Section \ref{sub} we deal with the subcritical regime, proving Theorem \ref{main1}. In Section \ref{super} we consider the supercritical regime, firstly in the balanced case in Section \ref{balanced}, and then the unbalanced case in Section \ref{unbalanced}. Finally in Section \ref{Dis} we discuss our results and give some open problems.

\section{Preliminaries}\label{Pre}

\subsection*{Notation}
Given a graph $G=(V,E)$ and a subset $A\subset V$, we will write $N(A)$ for the \emph{neighbourhood} of $A$, i.e., the set of vertices $w\in V$ such that there is some $v \in A$ with $ \{v,w\} \in E$. Note that, for a random bipartite graph, if $A \subseteq N_i$ then $N(A) \subseteq N_{3-i}$ for any $i =1,2$. A \emph{path} of length $n$, or an \emph{$n$-path}, is a graph $P$ on vertices $v_0,v_1,\dots,v_n$, for which $\{v_i,v_j\}\in E(P)$ if and only if $i-j=\pm 1$. We will write $P=v_0v_1 \ldots v_n$. We will write $x = (1 \pm \epsilon)y$ to mean $x \in [(1 - \epsilon)y,(1 + \epsilon)y]$.

Below we will state some useful results that we will need for our proofs. The first is a result of Scoins \cite{S62} which gives the number of spanning trees in a complete bipartite graph.

\begin{lemma}\label{bipartitesubgraphs}
For any $a,b\in \mathbb N$, the number of spanning trees of $K_{a,b}$ is $a^{b-1}b^{a-1}$.
\end{lemma}

We will also need to use the following result of Johansson (see \cite[Lemmas 7 and 8]{Johansson}) which concerns the component structure of $G(n_1,n_2,p)$ in the supercritical regime. In this paper, unless the base is explicitly mentioned, we will use $\log$ to denote the natural logarithm.
	\begin{theorem}\label{theorem3.2:Johansson}
	Let $G=G(n_1,n_2,p)$ with partition classes $N_1$ and $N_2$ of sizes $n_1$ and $n_2$, where $1\ll n_1\leq n_2$. If $p=\frac{d}{\sqrt{n_1n_2}}$ where $d >1$, then  there exist positive constants $\beta_0$,$\beta_1$ such that  whp in $G$
		\begin{itemize}
			\item at most one component  meets $N_1$ in more than $\beta_1\sqrt{n_1\log n_1}$  vertices;
			\item   no component meets $N_1$ in $k$ vertices with $k \in\left[\beta_0\log^2n_1, \beta_1\sqrt{n_1\log n_1}\right]$.
		\end{itemize}
	\end{theorem}
 The next result, which is a consequence of a theorem of Holley (see \cite[Section 1.3.2]{Johansson}), concerns stochastic domination of random variables. However, we will only state a specific case for binomial random graphs, as it is all we need. Firstly, we give a definition of stochastic domination for graph-valued random variables.

\begin{definition}\label{def2.6}
Let $G_1$ and $G_2$ be random variables which are distributed on the set of subgraphs of $K_n$. We say $G_1$ \emph{stochastically dominates} $G_2$, denoted by $G_2\preceq G_1$, if for every fixed graph $H$ 
$$
\mathbb{P}(G_2 \supseteq H) \leq \mathbb{P}(G_1 \supseteq H).
$$
\end{definition}
\begin{theorem}[Holley's Theorem]\label{Theorem3.5}
Let $G_1$ and $G_2$ be random variables which are distributed on the set of subgraphs of $K_n$. If 
    \begin{align*}\label{(21)}
        \mathbb{P}(e\in E(G_2)|E(G_2) -e = H) \leq \mathbb{P}(e\in E(G_1)|E(G_1)-e = H)
    \end{align*} 
   for all $e \in E(K_n)$ and for every subgraph $H$ of $K_n -e$, then $G_2\preceq G_1$.
\end{theorem}

In terms of probabilistic estimates, we will assume the reader is familiar with the use of Markov's and Chebyshev's inequalities, and we will use the following version of the Chernoff bounds, see for example \cite{Alon}.

\begin{theorem}[Chernoff bounds]\label{theorem2.8}
Let $X \sim Bin(n,p), \mu = \mathbb{E}(X) = np$, and let $t \geq 0$. Then,
\begin{itemize}
    \item $\mathbb{P}(X\geq \mathbb{E}(X)+t)\leq \exp \left(-\frac{t^2}{2\left(\mu+\frac{t}{3}\right)}\right)$;
    \item $\mathbb{P}(X\leq \mathbb{E}(X)- t)\leq \exp \left(-\frac{t^2}{2\mu}\right)$. 
\end{itemize}
\end{theorem}

\section{Subcritical regime: proof of Theorem \ref{main1}}\label{sub}
In this section we will prove Theorem \ref{main1}, which says that whp a random bipartite graph is planar in the subcritical regime. In fact, we will show slightly more. We say that a component of a graph is \emph{complex} if it contains more than one cycle and \emph{unicyclic} if it contains a unique cycle. We will show that in this range of $p$ whp  $G(n_1,n_2,p)$ contains no complex components.
\begin{lemma}\label{Lemma4.1}
Let $G=G(n_1,n_2,p)$ where $1\ll n_1\leq n_2$. If $p=\frac{d}{\sqrt{n_1n_2}}$ for constant $d<1$, then whp there is no complex component in $G$.
\end{lemma}
\begin{proof}
	Since containing a complex component is an increasing property, it is sufficient to prove the theorem for $p=\frac{1}{\sqrt{n_1n_2}}-\frac{n_2^{\frac{1}{6}}}{(n_1n_2)^{\frac{2}{3}}}$. In particular, we will show that, for this $p$, whp there is no complex component in $G$.
		
		If a component of $G$ contains at least two cycles, then in particular it has to contain a subgraph $H$ which consists of two cycles which are joined by a path, or which meet in a vertex, or which form a cycle with a diagonal path (see Figure \ref{H}). We note that each of these subgraphs can be constructed by taking a path $P=v_0v_1\ldots v_m$ and adding two edges of the form $\{v_0,v_i\}$ and $\{v_j,v_m\}$ with $0 < i,j < m$.
		Let $Y$ be the number of subgraphs of this type in $G$. Then the number of complex components is non-zero if and only if $Y$ is non-zero.

	In order to bound $\mathbb{E}(Y)$ we split into three cases, according to the length of the path $P$ and the partition class its initial vertex lies in. Since $G(n_1,n_2,p)$ is bipartite, the vertices of $P$ alternate between $N_1$ and $N_2$, and so, if the path has length $2k$ then it meets $k$ vertices in $N_1$ and $k$ vertices in $N_2$. For each choice of $2k$ vertices in this manner, there are $2(k!)^2$ possible paths on this vertex set in $K_{n_1,n_2}$, and for each path at most $k^2$ many different choices for the two extra edges to form $H$. Since $H$ has $2k+1$ edges, the probability that it is a subgraph of $G(n_1,n_2,p)$ is $p^{2k+1}$. Using a similar argument for the cases where the length of the path is odd, we see that
			\begin{align*}
		\mathbb{E}(Y)&\leq \sum_{k=3}^{n_1} \binom{n_1}{k}\binom{n_2}{k}2(k!)^2k^2p^{2k+1}+\sum_{k=2}^{n_1} \binom{n_1}{k}\binom{n_2}{k+1}(k+1)!k!k^2p^{2k+2}
		\\
		&\hspace{2cm}+\sum_{k=2}^{n_1} \binom{n_1}{k+1}\binom{n_2}{k}(k+1)!k!k^2p^{2k+2}\\
		&\leq 2\sum_{k=3}^{n_1}\frac{n_1^k}{k!}\frac{n_2^k}{k!}(k!)^2k^2 \frac{1}{(\sqrt{n_1n_2})^{2k+1}}\left(1-n_1^{-\frac{1}{6}}\right)^{2k+1}\\
		&\hspace{1cm}+\sum_{k=2}^{n_1}\frac{n_1^k}{k!}\frac{n_2^{k+1}}{(k+1)!}(k+1)!k!k^2 \frac{1}{(n_1n_2)^{k+1}}\left(1-n_1^{-\frac{1}{6}}\right)^{2k+2}\\
		&\hspace{1cm}+\sum_{k=2}^{n_1}\frac{n_1^{k+1}}{(k+1)!}\frac{n_2^{k}}{k!}(k+1)!k!k^2 \frac{1}{(n_1n_2)^{k+1}}\left(1-n_1^{-\frac{1}{6}}\right)^{2k+2}\\
		&\leq 4\sum_{k=3}^{n_1}\frac{k^2}{n_1}\exp{\left(-2kn_1^{-\frac{1}{6}}\right)}
		\\
		&\leq 4\int_{0}^{\infty}\frac{x^2}{n_1}\exp{\left(-2xn_1^{-\frac{1}{6}}\right)}dx\\
		&=\frac{1}{\sqrt{n_1}}\\
		&=o(1).
		\end{align*}
		By Markov's inequality, we can conclude that whp $Y=0$. Hence, whp there are no complex components in $G$. 
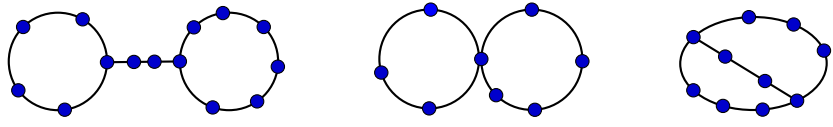
\begin{figure}[!ht]
	\center
\definecolor{ududff}{rgb}{0.30196078431372547,0.30196078431372547,1.}
\definecolor{qqqqff}{rgb}{0.,0.,1.}
\definecolor{qqqqcc}{rgb}{0.,0.,0.8}
\begin{tikzpicture}[line cap=round,line join=round,>=triangle 45,x=0.5cm,y=0.5cm]
\draw [line width=0.8pt] (3.,6.) circle (0.6462971452822611cm);
\draw [line width=0.8pt] (7.5,6.) circle (0.6463745044476924cm);
\draw [line width=0.8pt] (4.2924314850609395,5.9794852145228425)-- (6.207250991104615,6.);
\begin{scriptsize}
\draw [fill=qqqqcc] (3.18,4.72) circle (2.5pt);
\draw [fill=qqqqcc] (8.24,4.94) circle (2.5pt);
\draw [fill=qqqqcc] (4.2924314850609395,5.9794852145228425) circle (2.5pt);
\draw [fill=qqqqcc] (6.207250991104615,6.) circle (2.5pt);
\draw [fill=qqqqcc] (5.0001385562107545,5.987067369759658) circle (2.5pt);
\draw [fill=qqqqcc] (5.539862331166359,5.992849803782911) circle (2.5pt);
\draw [fill=qqqqcc] (3.6500373680768448,7.117251726382076) circle (2.5pt);
\draw [fill=qqqqcc] (2.085997811818812,6.914002188181188) circle (2.5pt);
\draw [fill=qqqqcc] (1.9588959049544783,5.233904533834427) circle (2.5pt);
\draw [fill=qqqqcc] (7.339654219791198,7.28276624167041) circle (2.5pt);
\draw [fill=qqqqcc] (6.576112661073385,6.904230161502644) circle (2.5pt);
\draw [fill=qqqqcc] (8.414111590562115,6.914111590562116) circle (2.5pt);
\draw [fill=qqqqcc] (7.071986687124215,4.780162058304016) circle (2.5pt);
\draw [fill=qqqqcc] (8.785317176159944,5.861581227182776) circle (2.5pt);
\end{scriptsize}
\end{tikzpicture}
\hspace{1cm}
\begin{tikzpicture}[line cap=round,line join=round,>=triangle 45,x=0.5cm,y=0.5cm]
\draw [line width=0.8pt] (3.8,3.) circle (0.6552098900352465cm);
\draw [line width=0.8pt] (6.5,2.98) circle (0.6646803743153548cm);
\begin{scriptsize}
\draw [fill=qqqqcc] (2.54,2.64) circle (2.5pt);
\draw [fill=qqqqcc] (5.56,2.04) circle (2.5pt);
\draw [fill=qqqqcc] (5.170774879666193,2.9989889302904826) circle (2.5pt);
\draw [fill=qqqqcc] (6.5,4.309360748630709) circle (2.5pt);
\draw [fill=qqqqcc] (3.8,1.689580219929507) circle (2.5pt);
\draw [fill=qqqqcc] (7.828750810195927,2.9397348239334566) circle (2.5pt);
\draw [fill=qqqqcc] (6.580419758784317,1.6530739800587704) circle (2.5pt);
\draw [fill=qqqqff] (3.8403015354809056,4.309799903129436) circle (2.5pt);
\end{scriptsize}
\end{tikzpicture}
\hspace{1cm}
\begin{tikzpicture}[line cap=round,line join=round,>=triangle 45,x=0.5cm,y=0.5cm]
\draw [rotate around={0.38712775415095124:(5.48,0.01)},line width=0.8pt] (5.48,0.01) ellipse (0.9631245871342223cm and 0.6164283983906508cm);
\draw [line width=0.8pt] (3.901734707673847,0.7104899078502777)-- (6.626780383375268,-0.9759965025522767);
\begin{scriptsize}
\draw [fill=qqqqcc] (5.36,1.24) circle (2.5pt);
\draw [fill=qqqqcc] (3.901734707673847,0.7104899078502777) circle (2.5pt);
\draw [fill=qqqqcc] (3.899253563914809,-0.7008156103799101) circle (2.5pt);
\draw [fill=qqqqcc] (5.721035336694717,-1.212221439903029) circle (2.5pt);
\draw [fill=qqqqcc] (6.626780383375268,-0.9759965025522767) circle (2.5pt);
\draw [fill=qqqqcc] (7.333670333527258,0.3525747204922167) circle (2.5pt);
\draw [fill=qqqqcc] (6.537054741592897,1.0448652733633943) circle (2.5pt);
\draw [fill=qqqqcc] (4.678541102525699,-1.114284421255369) circle (2.5pt);
\draw [fill=qqqqcc] (4.736286369246978,0.19399946933325007) circle (2.5pt);
\draw [fill=qqqqcc] (5.783607690864074,-0.45417064832080956) circle (2.5pt);
\end{scriptsize}
\end{tikzpicture}
	    \caption{Subgraphs witnessing complexity.}
	    \label{H}
	\end{figure}
\end{proof}
\begin{proof}[Proof of Theorem \ref{main1}]
    By Lemma \ref{Lemma4.1}, whp every component in $G$ is a tree or unicyclic, and therefore $G$ is planar. 
\end{proof}

\section{Supercritical regime}\label{super}
\subsection{The balanced case: proof of Theorems \ref{main2} and \ref{main4}}\label{balanced}
Recall that our plan is to use Euler's formula,
$$g(G)=\frac{1}{2}\left(e(G)-v(G)-f(G)+\kappa(G)+1\right),$$
to determine the genus of $G=G(n_1,n_2,p)$. Since $v(G)$ is fixed, and $e(G)$ is binomially distributed and hence tightly concentrated about its mean, it remains to estimate $f(G)$ and $\kappa(G)$.

Firstly, we will bound $f(G)$. We note that for our purposes in this section it would be sufficient to show that $f(G) = o(\sqrt{n_1n_2})$. However, since we will need a similar bound later, we will prove something slightly stronger below.

\begin{lemma}\label{l:cyclebound}
Let $G=G(n_1,n_2,p)$, where $1\ll n_1\leq n_2$, and $(n_1n_2)^{-\frac{1}{2}} \leq p$. Then for any function $h(n_1)= \omega(1)$  and $j\in \mathbb N$, whp
$$f(G)\leq h(n_1)j(p^2n_1n_2)^j + \frac{2}{j+1} \min \{ pn_1n_2,p^2n_1^2n_2\}.$$
	\end{lemma}
	\begin{proof}
	Note that a cycle of length $2k$ contains $k$ vertices in $N_1$ and $k$ vertices in $N_2$, and given such a set of $2k$ vertices there are $\frac{k!(k-1)!}{2}$ possible cycles of length $2k$ on them. Hence, if we let $f_1(G)$ be the number of faces of length at most $2j$ in $G$, then since each face is bounded by a cycle, and each cycle bounds at most two faces,
		\begin{align*}
		\mathbb{E}(f_1(G)) &\leq  2 \sum_{k=2}^{j} \binom{n_1}{k}\binom{n_2}{k}\frac{k!(k-1)!}{2}p^{2k}\\
		&\leq \sum_{k=2}^{j} \left(p^2n_1n_2\right)^k\\
		&\leq j\left(p^2n_1n_2\right)^j\hspace{1cm}\left(\text{since}\,\, p^2n_1n_2\geq 1\right).
		\end{align*}
		By Markov's inequality, we have that whp $f_1(G)$ is at most $h(n_1) j(p^2n_1n_2)^j$.
		
	Note that, since $e(G)$ is binomially distributed, it follows that whp $e(G) \leq 2pn_1n_2$. If we let $f_2(G)$ be the number of faces of length at least $2(j+1)$  then, since every edge is in the boundary of at most two faces, it follows that 
	\[
	f_2(G) \leq \frac{2e(G)}{2(j+1)}
	\]
    and hence whp $f_2(G) \leq \frac{2}{j+1}pn_1n_2$. However, if $n_1  \ll n_2$ then we can in fact bound this second term by a smaller quantity.
    
    We let 
    \[
    \mathcal{P} = \{ 2\text{-paths in } K_{n_1,n_2} \text{ with both endpoints in } N_1 \}.
    \]
    The idea is to count the number of 2-paths in $\mathcal{P}$ which appear in $G$. Let $S$  be the number of such 2-paths. We have
	 $$\mathbb{E}(S) =  \binom{n_1}{2}p^2n_2,$$
	 and
	 \begin{align*}
	     \text{Var}(S)&=\sum_{U\in \mathcal{P}}\mathbb{P}(U\in G)\sum_{\substack{V\in \mathcal{P},\\ V\ne U}}\left(\mathbb{P}(V\in G|U\in G)-\mathbb{P}(V\in G)\right)\\
	     &\leq  \mathbb{E}(S)\left(1+ 2p(n_1-2)\right).
	 \end{align*} 
	 Hence, by Chebyshev's inequality,
	 \[
	 \mathbb{P}(S \geq p^2n_1^2n_2) \leq \frac{1 + 2p(n_1-2)}{\binom{n_1}{2}p^2n_2} = o(1).
	 \]
		
		However, each $2$-path can be in the boundary of at most two faces, and each face of length at least $2(j+1)$ will contain at least $j+1$ many $2$-paths in $\mathcal{P}$. Hence, the number of faces of length at least $2(j+1)$ will satisfy
		\[
		    f_2(G)\leq \frac{2}{j+1} S,
		\]
		and so whp $f_2(G) \leq \frac{2}{j+1}p^2n_1^2n_2$.
		
It follows that whp
\[
f(G) = f_1(G) + f_2(G) \leq h(n_1) j(p^2n_1n_2)^j + \frac{2}{j+1} \min \left\{ pn_1n_2,p^2n_1^2n_2\right\},
\]
as claimed.
	\end{proof}

	\begin{corollary}\label{lemma8.1}
	Let $G=G(n_1,n_2,p)$, where $1\ll n_1\leq n_2$, and $p=\frac{d}{\sqrt{n_1n_2}}$ for constant $d>1$.  Then whp
		$$f(G)=o(n_1).$$
	\end{corollary}
	\begin{proof}
	We apply Lemma \ref{l:cyclebound} with $h(n_1) = j = \frac{1}{4} \log_d n_1$ to see that whp
	\begin{align*}
	f(G) &\leq h(n_1)j(p^2n_1n_2)^j + \frac{2}{j+1} p^2n_1^2n_2\\
	&\leq \frac{1}{16} \log_d^2 n_1  d^{\frac{1}{2}\log_d n_1} + \frac{8}{\log_d n_1 +1} d^2 n_1\\
	&\leq n_1 \left( \frac{\log_d^2 n_1}{16 \sqrt{n_1}} + \frac{8d^2}{\log_d n_1 +1} \right)\\
	&= o(n_1).
	\end{align*}
	
	\end{proof}
	
	 Next, we need to find $\kappa(G)$. By Theorem \ref{theorem3.2:Johansson}, all but one component $C$ of $G$ satisfies $|C\cap N_1|\leq  \beta_0\log^2n_1$ and so we only need to estimate the number of such components.
	\begin{definition}
	A component $C$ of $G(n_1,n_2,p)$ is \emph{small} if $|C\cap N_1|\leq  \beta_0\log^2n_1$, where $ \beta_0$ is as in Theorem~\ref{theorem3.2:Johansson}.
	\end{definition} 
	
	The following lemma, which follows from \cite[Lemma 2]{Johansson}, is a simple consequence of the Chernoff bounds.
	\begin{lemma}\label{lemma3.1}
		 Let $G=G(n_1,n_2,p)$ where $1 \ll n_1=\lambda n_2$ for constant $\lambda\leq 1$. If $p=\frac{d}{\sqrt{n_1n_2}}$ for constant $d>1$, then whp every small component $C$ in $G$ satisfies 
		 $$|C\cap N_2| = (1+o(1))pn_2|C \cap N_1|.$$
	\end{lemma}
	\begin{definition}
	A component $C$ of $G(n_1,n_2,p)$ is \emph{balanced} if $|C\cap N_2| \leq 2 pn_2|C \cap N_1|$.
	\end{definition} 
	Again, by Lemma \ref{lemma3.1} we only need to estimate the number of small balanced components.
	
	Let us first show that the number of small balanced unicyclic or complex components is negligible compared to the number of small balanced tree components.
	
	\begin{lemma}\label{lemma3.3}
Let $G=G(n_1,n_2,p)$ where $1 \ll n_1=\lambda n_2$ for constant $\lambda\leq 1$. If $p=\frac{d}{\sqrt{n_1n_2}}$ for constant $d>1$, then whp the number of small balanced unicyclic components in $G$ is $o(\log^5 n_1)$.
	\end{lemma}
	\begin{proof}
	Let us write $\mathcal{T}$ for the set of subgraphs $T$ of $K_{n_1,n_2}$ such that $T$ is a small balanced tree, i.e., $|T\cap N_1|\leq \beta_0\log^2n_1$ and $|T\cap N_2| \leq 2 pn_2|T\cap N_1|$. Similarly, let us write $\mathcal{U}$ for the set of small balanced unicyclic subgraphs $U$ of $K_{n_1,n_2}$. Note that every $U \in \mathcal{U}$ contains some $T \in \mathcal{T}$. Let us choose, for each $U \in \mathcal{U}$ one $T(U) \in \mathcal{T}$ such that $T(U) \subseteq U$. Note that for each $T \in \mathcal{T}$ there are at most
	\[
	|T \cap N_1| \cdot |T\cap N_2|\leq 2 \sqrt{\lambda} \beta_0^2 d \log^4 n_1 = o\left( \log^{5} n_1 \right)
	\]
	many $U$ such that $T = T(U)$.
	
	Now, for every $U \in \mathcal{U}$ it is clear that
	$$
	\mathbb{P}( U \text{ is a component in } G) = \frac{p}{1-p} \mathbb{P}(T(U) \text{ is a component in } G).
	$$
	 Hence, if we let $X_U$ be the number of small balanced unicyclic components and $X_T$ be the number of small balanced tree components in $G$, then we have
	\begin{align*}
	 \mathbb{E}(X_U) &\leq \sum_{U \in \mathcal{U}} \mathbb{P}(U \text{ is a component in } G)  \\
	 &=\frac{p}{1-p} \sum_{U \in \mathcal{U}} \mathbb{P}(T(U) \text{ is a component in } G)\\
	 &\le 2p \sum_{T \in \mathcal{T}}	|T \cap N_1| \cdot |T \cap N_2| \ \mathbb{P}(T \text{ is a component in } G) \\
	 &= o\left( \log^{5} n_1 \right)p\sum_{T \in \mathcal{T}} \mathbb{P}(T \text{ is a component in } G) \\
	 &= o\left( \log^{5} n_1 \right)p\mathbb{E}(X_T) \\
	 &\leq o\left( \log^{5} n_1 \right)p(n_1+n_2) \\
	 &= o\left( \log^{5} n_1 \right),
	\end{align*}
	since $X_T \leq n_1 + n_2$. It then follows by Markov's inequality that whp $X_U = o\left( \log^{5} n_1 \right) $.
		\end{proof}
		
		Using similar methods we can show that whp there are no small balanced {\em complex} components in $G$. However, since the calculations are similar to the above, we defer the proof of the following lemma to Appendix \ref{app}.

	\begin{lemma}\label{lemma8.2}
Let $G=G(n_1,n_2,p)$ where $1 \ll n_1=\lambda n_2$ for constant $\lambda\leq 1$. If $p=\frac{d}{\sqrt{n_1n_2}}$ for constant $d>1$, then whp there are no small balanced complex components in $G$.
	\end{lemma}

	 By Lemmas \ref{lemma3.3} and \ref{lemma8.2} it suffices to estimate the number of small balanced tree components in $G(n_1,n_2,p)$.

	 \begin{lemma}\label{lemma3.4}
	Let $G=G(n_1,n_2,p)$ where $1 \ll n_1=\lambda n_2$ for constant $\lambda\leq 1$. If $p=\frac{d}{\sqrt{n_1n_2}}$ for constant $d>1$, then whp the number of small balanced tree components in $G$ is \[(1+o(1))\nu(d,\lambda)n_1,\]
	where
		\begin{align}\label{e:nu}
\nu(d,\lambda)=\frac{1}{d\sqrt{\lambda}}\sum_{k=1}^{\infty}\left(\frac{d}{\sqrt{\lambda}}e^{-\frac{d}{\sqrt{\lambda}}}\right)^k\sum_{\substack{r+s=k,\\ 0\leq r,s\leq k}}\frac{r^{s-1}s^{r-1}}{r!s!}\lambda^re^{-\frac{d(\lambda-1)}{\sqrt{\lambda}}s}.
	   \end{align}
	\end{lemma}
	\begin{proof}
	Let $\mathcal{T}$ be as in Lemma \ref{lemma3.3} and let us denote by $X'_T$ to be the number of small balanced tree components in $G$. For any $T\in \mathcal{T}$ it follows that
	\[
	|T| = |T \cap N_1| + |T \cap N_2| \leq (1+2pn_2)\beta_0 \log^2 n_1 \leq \log^3n_1,
	\]
		when $n_1$ is large enough. Let $X_T$ be the number of tree components with $|T| \leq \log^3 n_1$. By Lemmas \ref{lemma3.3} and \ref{lemma8.2} whp $X_T = X'_T$.
		
		We can estimate that
		\[
			\mathbb{E}(X_T) =\sum_{k=1}^{\log^3n_1}\sum_{\substack{r+s=k,\\ 0\leq r,s\leq k}}\binom{n_1}{r}\binom{n_2}{s}r^{s-1}s^{r-1}p^{k-1}(1-p)^{r(n_2-s)+s(n_1-r)+rs -k+1},
		\]
		where the individual terms come from, for each $r+s =k \leq \log^3 n_1$, choosing a potential tree component with $|T \cap N_1| =r$ and $|T \cap N_2|=s$, which we can do in $\binom{n_1}{r}\binom{n_2}{s}r^{s-1}s^{r-1}$ many ways by Lemma \ref{bipartitesubgraphs}, each of which is a component of $G$ if and only if the $k-1$ edges in $T$ appear, and the $r(n_2-s)+s(n_1-r)+rs -k+1$ edges which have at least one endpoint in $V(T)$, but are not in $T$, do not appear in $G$. However, using the standard estimate that $(1-x) = e^{-x + O(x^2)}$ , we see that
		\begin{align*}
		(1-p)^{r(n_2-s)+s(n_1-r)+rs -k+1} &= \exp\left({-\left(n_2r+n_1s-rs-k+1\right)\left( p+O(p^2)\right)}\right)\\
		&=\exp\left({-\frac{d}{\sqrt{n_1n_2}}\left(n_2r+n_1s\right)+o(1)}\right)\\
		&= \exp\left({-\frac{d}{\sqrt{\lambda}}(r+\lambda s)+o(1)}\right).
		\end{align*}

It follows that $\mathbb{E}(X_T)$ is equal to
\small
		\begin{align*}
	&\sum_{k=1}^{\log^3n_1}\sum_{\substack{r+s=k,\\ 0\leq r,s\leq k}}\frac{(n_1)_r}{r!}\frac{(n_2)_s}{s!}r^{s-1}s^{r-1}\left(\frac{d}{\sqrt{n_1n_2}}\right)^{k-1}\exp\left({-\frac{d}{\sqrt{\lambda}}(r+\lambda s)+o(1)}\right)\nonumber \\ 
		&= \sum_{k=1}^{\log^3n_1}\sum_{\substack{r+s=k,\\ 0\leq r,s\leq k}} \frac{r^{s-1}s^{r-1}}{r!s!}d^{k-1}\frac{(n_1)_r(n_2)_s}{(\sqrt{n_1n_2})^{k-1}}\exp\left({-\frac{d}{\sqrt{\lambda}}(r+\lambda s)+o(1)}\right)\nonumber \\
		&=\sum_{k=1}^{\log^3n_1}\sum_{\substack{r+s=k,\\ 0\leq r,s\leq k}}\frac{r^{s-1}s^{r-1}}{r!s!}d^{k-1}\frac{(n_1)_r(n_2)_s}{(\sqrt{n_1n_2})^{k-1}}\exp\left({-\frac{d}{\sqrt{\lambda}}(r+\lambda s)+o(1)}\right)\nonumber \\
		&= \frac{1}{d}\sum_{k=1}^{\log^3n_1}\left(de^{-\frac{d}{\sqrt{\lambda}}}\right)^k \sum_{\substack{r+s=k,\\ 0\leq r,s\leq k}}\frac{r^{s-1}s^{r-1}}{r!s!} \frac{(n_1)_r(n_2)_s}{(\sqrt{n_1n_2})^{k-1}}\exp\left(-\frac{d}{\sqrt{\lambda}}(\lambda-1)s+o(1)\right), 
		\end{align*}
		\normalsize
		where we used that $r=k-s$ in the final line.
		
		Note that for any $r\leq \log^3n_1$
		\begin{align*}
		(n_1)_r&= n_1^r\prod_{t=0}^{r-1}\left(1-\frac{t}{n_1}\right)=n_1^r\exp{\left(\sum_{t=0}^{r-1}\log{\left(1-\frac{t}{n_1}\right)}\right)}\\
		&=n_1^r\exp{\left(\sum_{t=0}^{r-1}\left(-\frac{t}{n_1}-O\left(\frac{t^2}{n_1^2}\right)\right)\right)}\\
		&=n_1^r\exp{\left(-\frac{r(r-1)}{2n_1}-O\left(\frac{r^3}{n_1^2}\right)\right)}\\
		&=(1+o(1))n_1^r.
		\end{align*}
		Similarly, for any $s\leq \log^3n_1$,
		$(n_2)_s=(1+o(1))n_2^s$. Hence, using $n_1=\lambda n_2$ and $r+s=k$, we obtain
		\begin{align}
		\frac{(n_1)_r(n_2)_s}{(\sqrt{n_1n_2})^{k-1}}=(1+o(1))\frac{(\lambda n_2)^rn_2^s}{(n_2\sqrt{\lambda})^{k-1}}=(1+o(1))\frac{1}{(\sqrt{\lambda})^{k+1}}\lambda^r n_1.
		\end{align}
		Hence,
		\begin{align*}
		\mathbb{E}(X_T)=(1+o(1))\zeta(d,n_1,\lambda)n_1,
		\end{align*}
		where
		$$\zeta(d,n_1,\lambda)=\frac{1}{d\sqrt{\lambda}}\sum_{k=1}^{\log^3 n_1}\left(\frac{d}{\sqrt{\lambda}}e^{-\frac{d}{\sqrt{\lambda}}}\right)^k\sum_{\substack{r+s=k,\\ 0\leq r,s\leq k}}\frac{r^{s-1}s^{r-1}}{r!s!}\lambda^re^{-\frac{d(\lambda-1)}{\sqrt{\lambda}}s}.$$
		However, since $X_T\leq n_1 + n_2 = (1 + \frac{1}{\lambda})n_1$, $\mathbb{E}(X_T)\leq \left(1 + \frac{1}{\lambda}\right)n_1$. Hence, for large enough $n_1$, $\zeta(d,n_1,\lambda)\leq \frac{2}{\lambda}$. Furthermore, $\zeta(d,n_1,\lambda)\leq \zeta(d,n_1+1,\lambda)$  for any $n_1$. So that $\left(\zeta(d,n_1,\lambda)\right)_{n_1}$ is an increasing and dominated sequence. This implies that $\zeta(d,n_1,\lambda)$ converges to $\nu(d,\lambda)$ as $n_1 \rightarrow \infty$, where
		\begin{equation}\label{e:nubound}
		 \nu(d,\lambda)=\frac{1}{d\sqrt{\lambda}}\sum_{k=1}^{\infty}\left(\frac{d}{\sqrt{\lambda}}e^{-\frac{d}{\sqrt{\lambda}}}\right)^k\sum_{\substack{r+s=k,\\ 0\leq r,s\leq k}}\frac{r^{s-1}s^{r-1}}{r!s!}\lambda^re^{-\frac{d(\lambda-1)}{\sqrt{\lambda}}s} \leq \frac{2}{\lambda}.
		 \end{equation}
		Hence, 
		 $$\mathbb{E}(X_T)=(1+o(1))\nu(d,\lambda)n_1.$$
		 
	Now we will show that whp $X_T=(1+o(1))\mathbb{E}(X_T)$ by showing that 
		$$\text{Var}(X_T)=o\left(\mathbb{E}\left(X_T\right)^2\right).$$ 
		If we let $\mathcal{T}'$ be the set of all tree subgraphs of $K_{n_1,n_2}$ such that $|T| \leq  \log^3 n_1$, then we have
		\begin{align*}
		\text{Var}(X_T)&=\sum_{A\in \mathcal{T}'} \mathbb{P}(\mathds{1}_A=1)\left(\sum_{B\in \mathcal{T}'} \mathbb{P}(\mathds{1}_B=1 \mid \mathds{1}_A=1)-\mathbb{P}(\mathds{1}_B=1)  \right),
		\end{align*}
		where for any $C\in \mathcal{T}'$, $\mathds{1}_C$ is the indicator function of the event that $C$ is a component in $G$.
		Fix $A\in \mathcal{T}'$, for each $B\in \mathcal{T}'$, we split into the following three cases:
		
		In the first case, when $A=B$, 
		$$\mathbb{P}(\mathds{1}_B=1 \mid \mathds{1}_A=1)-\mathbb{P}(\mathds{1}_B=1) = 1-\mathbb{P}(\mathds{1}_A=1).$$
		
		In the second case, when $A\neq B$, $A$ and $B$ share at least one vertex: 
		$$\mathbb{P}(\mathds{1}_B=1 \mid \mathds{1}_A=1)-\mathbb{P}(\mathds{1}_B=1) = -\mathbb{P}(\mathds{1}_B=1),$$
		since it is impossible that both $B$ and $A$ are tree components if they share a vertex but are not identical.
		
		Finally, suppose that $A$ and $B$ are disjoint trees. Let us write $e(A,B)$ for the number of edges between the vertex sets of $A$ and $B$ in $K_{n_1,n_2}$ and let us write $\phi(B)=e(K_{n_1n_2}-E(B))$. By definition of $\mathcal{T}'$, $e(A,B)= O\left(\log^6n_1\right)$, and since $1-p\geq e^{-2p}$ for small $p$, when $n_1$ is large enough we have
		\begin{align*}
		&\mathbb{P}(\mathds{1}_B=1 \mid \mathds{1}_A=1)-\mathbb{P}(\mathds{1}_B=1)\\ &=\mathbb{P}(\mathds{1}_B=1)\left(\frac{\mathbb{P}(\mathds{1}_B=1\mid \mathds{1}_A=1)}{\mathbb{P}(\mathds{1}_B=1)}-1\right) \\&= \mathbb{P}(\mathds{1}_B=1)\left(\frac{p^{k-1} (1-p)^{\phi(B)- e(A,B)}}{p^{k-1}(1-p)^{\phi(B)}}-1\right)\\
		&= \mathbb{P}(\mathds{1}_B=1)\left((1-p)^{-e(A,B)}-1\right)\\
		&\leq \mathbb{P}(\mathds{1}_B=1)\left(e^{2\frac{d}{\sqrt{n_1n_2}}e(A,B)}-1\right)\\
		&=o(1)\mathbb{P}(\mathds{1}_B=1) \hspace{1cm}\left(\text{since}\,\,e(A,B)=O\left(\log^6{n_1}\right)\right).
		\end{align*}
		
		From all of above cases, we conclude
		\begin{align*}
		\text{Var}(X_T)&\leq \sum_{A\in \mathcal{T}}\mathbb{P}(\mathds{1}_A=1)\left(1+o(1)\mathbb{E}(X_T)\right)\\
		&=\mathbb{E}(X_T)+ o\left(\left(\mathbb{E}(X_T)\right)^2\right)\\
		&= o\left(\left(\mathbb{E}(X_T)\right)^2\right),
		\end{align*}
		since $\mathbb{E}(X_T)$ tends to $\infty$. It follows that whp
		$$X_T=(1+o(1))\mathbb{E}(X_T)=(1+o(1))\nu(d,\lambda)n_1.$$
		
		Finally, since whp $X'_T = X_T$, the result follows.
	\end{proof}

  As a corollary we get the following bound on $\kappa(G)$.
  \begin{corollary}\label{c:kappabound}
  Let $G=G(n_1,n_2,p)$ where $n_1=\lambda n_2$ for constant $\lambda\leq 1$. If $p=\frac{d}{\sqrt{n_1n_2}}$ for constant $d>1$, then whp
  $$\kappa(G)=(1+o(1))\nu(d,\lambda)n_1.$$
  \end{corollary}
  \begin{proof}
  By Theorem \ref{theorem3.2:Johansson} and Lemma \ref{lemma3.1}, whp $\kappa(G)$ is the number of small balanced components. By Lemmas \ref{lemma3.3}, \ref{lemma8.2},  and  \ref{lemma3.4}, whp the number of small balanced components in $G$ is  $$o\left(\log^5n_1\right)+0+(1+o(1))\nu(d,\lambda)n_1=(1+o(1))\nu(d,\lambda)n_1.$$ 
\end{proof}
The proof of Theorem \ref{main2} then follows.
\begin{proof}[Proof of Theorem \ref{main2}]
Since $e(G)$ is binomially distributed it follows that whp $e(G) =(1+o(1)) pn_1n_2 = (1+o(1))\frac{d}{\sqrt{\lambda}} n_1$. Furthermore $v(G) = n_1 + n_2 = \left(1+\frac{1}{\lambda}\right)n_1$. By Corollary \ref{lemma8.1}, we have that whp $f(G) = o(n_1)$ and by Corollary \ref{c:kappabound} whp $\kappa(G)=(1+o(1))\nu(d,\lambda)n_1$. 

Therefore, using Euler's formula, we have
 \begin{align*}
     g(G)&=\frac{1}{2}\left(e(G)-v(G)-f(G)+\kappa(G)+1\right)\\
     &=\frac{1}{2}\left((1+o(1))\frac{d}{\sqrt{\lambda}} n_1-\left(1+\frac{1}{\lambda}\right)n_1+o(n_1)+(1+o(1))\nu(d,\lambda)n_1+1\right)\\
     &=\frac{1}{2}\left((1+o(1))\frac{d}{\sqrt{\lambda}}-\left(1+\frac{1}{\lambda}\right)+(1+o(1))\nu(d,\lambda)+o(1)\right)\frac{\sqrt{\lambda}}{d}pn_1n_2\\
     &=(1+o(1))\gamma(d,\lambda)pn_1n_2,
 \end{align*}
 where $\gamma(d,\lambda)= \frac{1}{2}-\frac{\lambda+1}{2d\sqrt{\lambda}}+\frac{\nu(d,\lambda)\sqrt{\lambda}}{2d}.$ This completes the proof.
\end{proof}
\begin{remark}
Let us briefly compare this to Theorem \ref{(1)and(2)}. If we take $\lambda=1$ and $n_1 = n_2 = \frac{n}{2}$, then $G(n_1,n_2,p)$ is the subgraph of the random graph $G(n,p)$ obtained by deleting the edges inside the partition classes. In this case, we have that when $d>1$, whp
\[
g\left(G\left(\frac{n}{2},\frac{n}{2},\frac{d}{n}\right)\right) = (1+o(1))\gamma\left(\frac{d}{2},1\right)\frac{dn}{4} = (1+o(1))\left( \frac{1}{2}+ O \left(\frac{1}{d}\right)\right) \frac{dn}{4},
\]
since $\nu\left(\frac{d}{2},1\right) = O(1)$ by \eqref{e:nubound}. A similar argument will show that the number of components in $G(n,\frac{d}{n})$ is approximately
\[
\frac{n}{d}\sum_{k=1}^\infty\left(d e^{-d}\right)^k \frac{k^{k-2}}{k!} = O(n),
\]
and so
	\[
g\left(G\left(n,\frac{d}{n}\right)\right) = (1+o(1))\mu(d)\frac{dn}{2} =  (1+o(1))\left(\frac{1}{2}+O\left(\frac{1}{d}\right)\right) \frac{dn}{2}.
\]
Hence for large enough $d$, whp the genus of $G(n_1,n_2,p)$ will be approximately half of the genus of $G(n,p)$. \end{remark}

Finally in this section, we will prove Theorem \ref{main4}. We first note that following consequence of Lemma \ref{l:cyclebound}

\begin{corollary}\label{corollary5.2}
      Let $G=G(n_1,n_2,p)$ where $1 \ll n_1=\lambda n_2$ for constant $\lambda\leq 1$. If $(n_1n_2)^{-\frac{1}{2}} \ll p \ll\left(n_1n_2\right)^{-\frac{1}{2}+o(1)}$, then whp
       $$f(G)= o(pn_1n_2).$$
       \end{corollary}
       \begin{proof}
       We will show that, given any constant $\epsilon>0$, whp 
       $$f(G)\leq \epsilon pn_1n_2.$$
       
       We choose $j\in \mathbb{N}$ such that $j \geq\frac{4}{\epsilon} -1$, so that $\frac{1}{j+1}\leq \frac{\epsilon}{4}$. Note that, since $p \ll (n_1n_2)^{-\frac{1}{2} + \frac{1}{8j}}$ it follows that
       \[
       (p^2n_1n_2)^j = o\left((n_1n_2)^{\frac{1}{4}}\right).
       \]
       
       Hence, by Lemma \ref{l:cyclebound} with $h(n_1) = n_1^{\frac{1}{4}}$, we see that whp
       \begin{align*}
       f(G) &\leq h(n_1)j(p^2n_1n_2)^j + \frac{2}{j+1} pn_1n_2 \\
       &\leq n_1^{\frac{1}{4}} j (p^2n_1n_2)^j + \frac{2}{j+1} pn_1n_2\\
       &\leq o\left( (n_1n_2)^{\frac{1}{2}}\right) + \frac{\epsilon}{2}pn_1n_2\\
       &= o(pn_1n_2) + \frac{\epsilon}{2}pn_1n_2\\\
       &\leq \epsilon pn_1n_2.
       \end{align*}
       \end{proof}
       
   Now we complete this section by proving Theorem \ref{main4}.

\begin{proof}[Proof of Theorem \ref{main4}]
By Euler's formula
$$g(G)=\frac{1}{2}\left(e(G)-v(G)-f(G)+\kappa(G)+1\right),$$
and so the upper bound then follows from the fact that $v(G)\geq \kappa(G)$ and $f(G)\geq 1$.

Conversely, we can also deduce the lower bound from Euler's formula. Indeed, since whp $e(G)=(1+o(1))pn_1n_2$, it follows that $v(G)=n_1+n_2=o(pn_1n_2)$ as $n_1=\Theta(n_2)$, and by Corollary \ref{corollary5.2} whp  $f(G)=o(pn_1n_2)$, and so
		$$g(G)\geq \frac{1}{2}\left(e(G)-v(G)-f(G)\right)= (1-o(1))\frac{1}{2}pn_1n_2,$$
		 and we are done.
\end{proof}

\subsection{The unbalanced case: proof of Theorem \ref{main3}}\label{unbalanced}
In this section we consider the unbalanced case, where the first partition class is much smaller than the second. 

In order to bound the genus in this case, we consider an auxilliary graph $H$, which we call the \emph{$2$-centre} of $G$, defined as follows:

Given $G=G(n_1,n_2,p)$ we first construct a
graph $G'$ by deleting all the vertices in $N_2$ whose degree in $G$ is not equal to two. Then $H$ is the graph with vertex set $V(H) = N_1$ and edge set 
\[
E(H) = \{ \{x,y\} \colon \text{ there exists a $2$-path from $x$ to $y$ in $G'$}\}.
\]

We note that  
$H$ is similar to a graph considered by Johansson (see \cite[chapter 3]{Johansson}), which he called the even projection of $G$. 

We will show two useful facts about $H$. Firstly,  the genus of $H$ is approximately the genus of $G$ and secondly, $H$ is distributed approximately like a supercritical binomial random graph. Together with Theorem \ref{(1)and(2)} these facts will be enough to determine the genus of $G$.

 \begin{lemma}\label{lemma5.1}
 Let $G=G(n_1,n_2,p)$ and $H$ be the $2$-centre of $G$.  Then
		$$g(H)\leq g(G).$$
	\end{lemma}
	\begin{proof}
		Since $H$ is a minor of $G$, $g(H)\leq g(G)$.
	\end{proof}
	\begin{lemma}\label{lemma5.2}
	 Let $G=G(n_1,n_2,p)$ with $1 \ll n_1 \ll n_2$ and $p \ll n_1^{-1}$. Let $H$ be the $2$-centre of $G$. Then whp 
		$$g(G)\leq g(H)+O(p^3n^3_1n_2).$$
	\end{lemma}
	\begin{proof}
		Recall that $H$ can be obtained from $G$ via the following steps:
		\begin{itemize}
			\item Let $H_1=G-V_1$ where $V_1$ is the set of all vertices of degree $0$ or $1$ in $N_2$.
			
			\item Let $H_2=H_1-V_2$ where $V_2$ is the set of all vertices of degree at least $3$ in $N_2\setminus V_1$. 
			
			\item  Let $H_3$ be the multigraph formed by adding an edge $\{x,y\}$ for every 2-path $xzy$ in $H_2$ with $x,y \in N_1$, and then deleting $N_2$. Note that $V(H_3)=N_1$ and that $H_3$ may contain parallel edges.
			
			\item Let $H$ be obtained from $H_3$ by replacing each set of parallel edges by a single edge. 
		\end{itemize}
		
		In the above process, we note that $g(H_1)=g(G)$, since deleting isolated vertices and leaves preserves the genus. Then, if we let $Z$ be the number of edges adjacent to vertices in $V_2$ we see that 
		\begin{align*}
		    g(H_1)\leq g(H_2)+Z.
		\end{align*}
		Indeed, we cannot increase the genus of a graph by more than one by adding a single edge, and so the claim follows inductively.
		
		Furthermore $H_2$ is a subdivision of $H_3$ and so $g(H_3)=g(H_2)$, and finally adding parallel edges does not change the genus of a graph and so $g(H)=g(H_3)$. 
		
		Putting this all together, we conclude that 
		$$g(G)\leq g(H)+Z.$$
		Hence, in order to complete the proof, we need to show that whp
		$$Z=O(p^3n^3_1n_2).$$
		Indeed, if we fix $x\in N_1$ and $y\in N_2$, then 
		\begin{align*}
		&\mathbb{P}(\text{$\{x,y\}$ is an edge and $y$ has degree at least 3})\\
		&=p\sum_{k=2}^{n_1-1} \binom{n_1-1}{k}p^k(1-p)^{n_1-1-k}\\
		& =p\left(1-(1-p)^{n_1-1}-(n_1-1)p(1-p)^{n_1-2}\right)\\
		&\leq p\left(p(n_1-1)-(n_1-1)p(1-p)^{n_1-2}\right)\\
		&=p^2(n_1-1)\left(1-(1-p)^{n_1-2}\right)\\
		&\leq p^3n_1^2.
		\end{align*}
		Hence,  
		$$\mathbb{E}(Z)\leq n_1n_2p^3n_1^2 = p^3n_1^3n_2.$$
		
		Furthermore, it is a simple calculation that Var$(Z) \leq \mathbb{E}(Z)pn_1 = o\left(\mathbb{E}(Z) \right)$ and hence, by Chebyshev's inequality, whp $Z = (1+o(1)) \mathbb{E}(Z) = O(p^3n_1^3n_2)$.
			\end{proof}
	Next, we prove Lemma \ref{lemma5.3} regarding the distribution of the \emph{$2$-centre} $H$ of $G(n_1,n_2,p)$.

	\begin{proof}[Proof of Lemma \ref{lemma5.3}]
		
	Fix an arbitrary  edge $\{x,y\}\in [n_1]^2$ and let $A_e$ be the event that $e:=\{x,y\}\in H$. Let us define the random set $X=\bigcup_{v\neq x,y\in N_1} N(v)$. 
	
	 We claim that; if we condition on the value of $X$, then the random variable $E(H) - e$ is independent of the event $A_e$.
	
	To see this, we note that conditioned on the event that $X=M$ for some $M \subseteq N_2$, the event $A_e$ only depends on the edges between $x,y$ and $N_2 \setminus M$ and the event $E(H)\setminus e=K$ only depends on the edges between $N_1$ and $M$, and since these edges sets are disjoint, the two events are independent. It follows that, for any $K$ and $M$
	$$\mathbb{P}(E(H)- e = K | A_e\,\, \text{and}\,\, X=M) = \mathbb{P}(E(H) - e = K | X=M).$$
	Hence, we have
		\begin{align*}
	    &\mathbb{P}(A_e\mid X=M\,\, \text{and}\,\, E(H) - e=K)\\
	    &= \mathbb{P}(A_e\mid X=M)\frac{\mathbb{P}(E(H)- e=K\mid A_e \,\,\text{and}\,\, X=M)}{\mathbb{P}(E(H)- e=K\mid X=M)}\\&= \mathbb{P}(A_e\mid X=M),
	\end{align*} 
	in other words, conditioned on the event that $X=M$, $E(H)-e$ is independent of the event $A_e$, and we have the claim.
	
	Note that, since $pn_1 = o(1)$, we have 
	$$\mathbb{E}(|X|)= n_2(1-(1-p)^{n_1-2})\leq pn_1n_2 = o(n_2),$$ and  $|X|$ is binomially distributed. Hence, for any positive $\gamma$, by Theorem \ref{theorem2.8}
	\begin{align}
	    \mathbb{P}(|X|-\mathbb{E}(|X|)\geq \gamma n_2) &\leq \text{exp}\left(-\frac{\gamma^2n_2^2}{2\left(pn_1n_2+\frac{\gamma n_2}{3}\right)}\right)  \nonumber\\
	    &\leq e^{-\gamma n_2} \nonumber\\
	    &\leq e^{-n_1} \hspace{1cm} (\text{since \,$n_1=o(n_2)$}) \label{5}.
	\end{align}
	Therefore, whp $|X|\leq  \mathbb{E}(|X|)+\gamma n_2 \leq 2\gamma n_2$. It follows that
	\begin{align*}
	&\mathbb{P}(A_e | E(H) \setminus e=K) \\
	&= \sum_{M \in 2^{N_2}} \mathbb{P}(A_e | E(H) \setminus e=K \text{ and } X=M)\mathbb{P}(X=M)\\
	&= O(e^{-n_1}) + \sum_{M \in 2^{N_2}, |M|\leq 2\gamma n_2} \mathbb{P}(A_e | E(H) \setminus e=K \text{ and } X=M)\mathbb{P}(X=M)\\
		&=O(e^{-n_1}) + \sum_{M \in 2^{N_2}, |M|\leq 2\gamma n_2} \mathbb{P}(A_e | X=M)\mathbb{P}(X=M),
	\end{align*}
	where the third line holds by \eqref{5}.
		However, for any fixed $|M|\leq 2\gamma n_2$, 
\begin{align*}
	    \mathbb{P}(A_e | X=M)=1-(1-p^2)^{n_2-|M|}&=p^2(n_2-|M|)+O\left(p^4n_2^2\right)\\
	     &=(1 \pm 2\gamma)p^2n_2,
\end{align*}
and so
	\begin{align}
	\mathbb{P}(A_e | E(H) \setminus e=K) &=  O(e^{-n_1}) + (1 \pm 2\gamma)p^2n_2\sum_{M \in 2^{N_2}, |M|\leq \sigma n_2}\mathbb{P}(X=M) \nonumber\\
	 &= (1 \pm 3\gamma)p^2n_2, \label{6}
	\end{align}
	since $p^2n_2 \geq n_1^{-1} \gg e^{-n_1}$.
	
	Since the appearance of edges in a binomial random graph are i.i.d., it follows that 
	\begin{equation}\label{lower}
	\noeqref{lower}
	    \mathbb{P}(e\in G_1 | E(G_1) \setminus e=K) = (1-\delta)\frac{d^2}{n_1},
	\end{equation}
	\begin{equation}\label{upper}
	    \mathbb{P}(e\in G_2 | E(G_2) \setminus e=K) = (1+\delta)\frac{d^2}{n_1}.
	\end{equation}
    So, if we choose $\gamma$ sufficiently small in terms of $\delta$, it follows from \eqref{6}--\eqref{upper} that
	\begin{align*}
	    \mathbb{P}(e \in G_1 | E(G_1) \setminus e=K)&\leq  \mathbb{P}(A_e | E(H) \setminus e=K)\\
	    &\leq \mathbb{P}(e=\{x,y\} \in G_2 | E(G_2) \setminus e=K).
	\end{align*}
Hence, since $e$ was chosen arbitrarily,  by Theorem \ref{Theorem3.5} (Holley's Theorem), whp $G_1\preceq H\preceq G_2$, and the result follows.
	\end{proof}
	We can now prove Theorem \ref{main3}.
	\begin{proof}[Proof of Theorem \ref{main3}]
	
Consider $G_1$ and $G_2$ as in Lemma \ref{lemma5.3}, with $\delta$ sufficiently small such that  $(1-\delta)d^2>1$. By Theorem \ref{(1)and(2)}, we have that whp
$$g(G_1)= (1+o(1))\mu\left((1-\delta)d^2\right)\frac{d^2n_1}{2};$$
$$g(G_2)=(1+o(1))\mu\left((1+\delta)d^2\right)\frac{d^2n_1}{2}.$$
	Since $G_1\preceq H\preceq G_2$, and genus is an increasing  graph parameter, it follows that whp $g(G_1)\leq g(H)\leq g(G_2)$. Moreover, $\mu$ is clearly an increasing function, which is continuous by Theorem \ref{(1)and(2)}.
	Hence, since we can choose $\delta$ arbitrarily small, whp
	\begin{align}\label{(9)}
	    g(H)=(1+o(1))\mu\left(d^2\right)\frac{d^2n_1}{2}.
	\end{align}
	However, by Lemmas \ref{lemma5.1} and \ref{lemma5.2}, whp	 
	\begin{align}\label{(20)}
	    g(H)\leq g(G)\leq g(H)+o(n_1).
	\end{align}
	Finally, by \eqref{(9)} and \eqref{(20)}, we conclude that whp
$$g(G)=(1+o(1))\mu\left(d^2\right)\frac{d^2n_1}{2},$$
and we are done.
	\end{proof}
\begin{remark}\label{r:unbalanced}
We note that, depending on the precise relationship between $n_1$ and $n_2$, our results show slightly more. Indeed, from Lemmas \ref{lemma5.1}, \ref{lemma5.2} and \ref{lemma5.3} we can conclude that, as long as $(n_1n_2)^{-\frac{1}{2}} \leq p \ll \min \left\{ n_1^{-1},n_2^{-\frac{1}{2}} \right\}$, $H$ has a distribution very close to that of $G(n_1,p^2n_2)$, and the genus of $G(n_1,n_2,p)$ is within $O(p^3n_1^3n_2)$ of the genus of $H$.

Note that, in this range of $p$, $\frac{1}{n_1} \leq p^2n_2 \ll \min\left\{ \frac{n_2}{n_1^2}, 1 \right\}$, and so if 
\[
n_1^{-\frac{i}{i+1}} \ll p^2n_2 \ll n_1^{-\frac{i-1}{i}} \quad (\text{for some } i \in \mathbb{N})
\]
then we can determine asymptotically the genus of $H$, which will be $\Theta(p^2n^2_1n_2)$ by the result of R\"{o}dl and Thomas \cite{Rodl}. Note that, if $p^2n_2 = c n_1^{-\frac{i}{i+1}}$ then the precise value of $g(G(n_1,p^2n_2))$ is not known, although it is still known that it is $\Theta(p^2n^2_1n_2)$, and hence so is $g(H)$.

Therefore, in all these cases, since $p^3n_1^3n_2 = o(p^2n^2_1n_2)$, it follows that 
\[
\text{whp the genus of $G(n_1,n_2,p)$ is $\Theta(p^2n^2_1n_2) = o(pn_1n_2)$,}
\]
 and, in particular, sublinear in the number of edges.

So, for example, as an analogue to Theorem \ref{theorem1.2}, if $n_2 \geq n_1^2$ then as $p$ increases from $(n_1n_2)^{-\frac{1}{2}}$ to $n_2^{-\frac{1}{2}}$ the genus will be linear in $p^2n^2_1n_2$, with phase transitions for the leading constant occurring at 
\[
p = \Theta \left( n_1^{-\frac{i}{2(i+1)}}n_2^{-\frac{1}{2}} \right) \quad (\text{for some } i \in \mathbb{N}).
\]
If $n_1 \ll n_2 \ll n_1^2$ is smaller, then we can only conclude that the same behaviour holds up to $p = n_1^{-1}$.
\end{remark}
\begin{remark}
Similarly, if we have some sequences $n_1(n)$ and $n_2(n)$ and $p(n) = \left(n_1 n_2 \right)^{-\frac{1}{2} + \epsilon}$, where $\epsilon = \epsilon(n)= o(1)$ is some positive function tending to $0$, then as long as $n_1^{\frac{1-2\epsilon}{2\epsilon}} \gg n_2 \gg n_1^{\frac{1 + 2\epsilon}{1- 2\epsilon}}$, it follows that
\[
\left(n_1n_2\right)^{-\frac{1}{2}} \leq p \ll \min \left\{ n_1^{-1}, n_2^{-\frac{1}{2}}\right\}
\]
and so we can conclude from Lemmas \ref{lemma5.1}, \ref{lemma5.2} and \ref{lemma5.3} that $H$ has a distribution very close to that of $G(n_1,p^2n_2)$, and the genus of $G(n_1,n_2,p)$ is within $O(p^3n_1^3n_2)$ of the genus of $H$. Note again that in this range of $p$, $p^3n_1^3n_2 = o(p^2n_1^2n_2)$.

In particular, if $n_2^{2\epsilon} = n_1^{o(1)}$, then we can conclude from Theorem \ref{(1)and(2)} that whp $g(G(n_1,p^2n_2)) = (1+o(1))\frac{p^2 n_1^2 n_2}{4}$. It follows that also
\[
\text{whp the genus of $G(n_1,n_2,p)$ is $(1+o(1))\frac{p^2 n_1^2 n_2}{4}$,}
\]
 giving us an analogue of Theorem \ref{main4} when $n_1 \ll n_2$.

In particular, we note that the conditions on $n_1$ and $n_2$ are satisfied whenever $n_2 = n_1^c$ for some positive $c >1$. 
\end{remark}
\section{Discussion}\label{Dis}
We have demonstrated that a sharp threshold for planarity occurs in the binomial random bipartite graph  $G(n_1,n_2,p)$ at $p=\frac{1}{\sqrt{n_1n_2}}$, but  that the behaviour of the genus in the supercritical regime differs between the balanced and the unbalanced case.

Let us mention a few open problems. In \cite{Kang} the authors give results for the genus of $G(n,p)$ in the \emph{weakly supercritical regime}.
\begin{theorem}[\cite{Kang}]
If $p = \frac{1+ \epsilon}{n}$ with $\epsilon \rightarrow 0$ and $\epsilon^3 n \rightarrow \infty$, then whp
\[
g(G(n,p)) = (1+o(1))\frac{\epsilon^3 n}{3}.
\]
\end{theorem}
It would be interesting to know how the genus of $G(n_1,n_2,p)$ behaves in this regime in the balanced case.
\begin{question}
Is there some constant $c(\lambda)$ such that if $n_1 = \lambda n_2$ with $\lambda\leq 1$ and $p = \frac{1+ \epsilon}{\sqrt{n_1n_2}}$ with $\epsilon \rightarrow 0$ and $\epsilon^3 \sqrt{n_1n_2} \rightarrow \infty$, then whp
\[
g(G(n_1,n_2,p)) = (1+o(1))c(\lambda)\epsilon^3 \sqrt{n_1n_2} ?
\]
\end{question}

In the balanced case, our results extend the results of Jing and Mohar \cite{Mohar}, showing that the genus of $G(n_1,n_2,p)$ is linear in the number of edges in the supercritical regime and above. In the unbalanced case, Jing and Mohar showed further that the genus of $G(n_1,n_2,p)$ is whp linear in the number of edges if $p \gg n_1^{-\frac{2}{3}}$. So, in the unbalanced case, there is a gap in our understanding between the `dense' case and the supercritical case covered in this paper.

Euler's formula, together with calculations as in \eqref{e:isolated} of the number of isolated components, suggests that $g(G(n_1,n_2,p)) \lesssim O(p^2n_1^2n_2)$ as long as $p \ll n_1^{-1}$, and so the genus will be sublinear in the number of edges. As mentioned in Remark \ref{r:unbalanced}, this will indeed be the case if $n_2 \gg n^2_1$. It would be interesting the know what happens for smaller $n_2$.
\begin{question}
What is the correct order of growth of the genus of $G(n_1,n_2,p)$ when $n_1 \ll n_2 \ll n_1^2$ and $n_2^{-\frac{1}{2}} \ll p \ll n_1^{-1}$?
\end{question}

Furthermore, it would be interesting to know how the genus behaves in this gap between $p=n_1^{-1}$ and $n_1^{-\frac{2}{3}}$, and at which point the genus becomes linear in the number of edges.
\begin{question}
Is the genus of $G(n_1,n_2,p)$ sublinear or linear in the number of edges when $n_1 \ll n_2$ and $n_1^{-1} \ll p \ll n_1^{-\frac{2}{3}}$?
\end{question}

Finally, it would be nice to know the asymptotics of $\gamma(d,\lambda)$ (or equivalently $\nu(d,\lambda)$) as $d \rightarrow \infty$ or $d \rightarrow 1$, as with Theorem \ref{(1)and(2)}. It is clear that as $d \rightarrow \infty$, $\gamma(d,\lambda) \rightarrow \frac{1}{2}$, and, since the sum in \eqref{e:nu} is convergent for $d>1$, the function $\gamma(d,\lambda)$ is continuous for $d \in (1,\infty)$. When $d <1$ it follows from work of Johansson \cite{Johansson} that whp there is no giant component, and so standard arguments give that the number of components of $G$ is asymptotically equal to the excess of $G$, which implies that
\[
\nu(d,\lambda)= 1+\frac{1-d\sqrt{\lambda}}{\lambda},
\]
and furthermore, that \eqref{e:nu} still holds for $d \in (0,1)$. This suggests that the limit of $\gamma(d,\lambda)$ as $d$ tends to $1$ should be $0$, and indeed when $d=\lambda=1$ it is easy to verify that the sum in \eqref{e:nu} converges, and so the limit is indeed $0$. However for $\lambda < 1$ the asymptotics of \eqref{e:nu} are more complicated, and we could not prove convergence.
\begin{question}
Is $\gamma(d,\lambda)$ continuous on $(0,\infty)$?
\end{question}
\bibliographystyle{plain}
\bibliography{Journalreferences}
\appendix
\section{The proof of Lemma \ref{lemma8.2}}\label{app}
\begin{proof}
Let $\mathcal{T}$ be as in Lemma \ref{lemma3.3} and let $\mathcal{C}$ be the set of subgraphs $C$ of $K_{n_1,n_2}$ such that $C$ contains exactly two cycles, $|C \cap N_1|\leq \beta_0\log^2n_1$, and $|C\cap N_2| \leq 2 pn_2|C\cap N_1|$. Note that every $C \in \mathcal{C}$ contains some $T \in \mathcal{T}$. Let us choose, for each $C \in \mathcal{C}$ one $T(C) \in \mathcal{T}$ such that $T(C) \subseteq C$. Note that for each $T \in \mathcal{T}$ there are at most 
\[
(|T \cap N_1| \cdot |T\cap N_2|)^2 \leq 4p^2n_2^2\beta_0^2 \log^8 n_1 = o(\log^9 n_1)
\] 
many $C$ such that $T = T(C)$.

Now, for every $C \in \mathcal{C}$ with $t$ vertices we have
	\begin{align*}
	\mathbb{P}(C \text{ spans a component of } G) &= p^2(1-p)^{\binom{t}{2} - t-1}\mathbb{P}(T(C) \text{ is a component in }G)\\
	&\leq p^2e^{-pt^2}\mathbb{P}(T(C) \text{ is a component in }G)\\
	&\leq 2p^2\mathbb{P}(T(C) \text{ is a component in }G),
	\end{align*}
	since $t \leq \beta_0 \log^2 n_1(1 + 2pn_2) \leq \log^3 n_1$.
	
	Furthermore, every small balanced complex component contains a spanning subgraph in $C$. Hence, if we let $X_C$ be the number of small balanced complex components and $X_T$ be the number of small balanced tree components in $G$, then we have
	\begin{align*}
	 \mathbb{E}(X_C) &\leq \sum_{C \in \mathcal{C}} \mathbb{P}(C \text{ spans a component in } G) \\
	 &=2p^2 \sum_{C \in \mathcal{C}} \mathbb{P}(T(C) \text{ is a component in } G) \\
	 &\leq 2p^2 \sum_{T \in \mathcal{T}}(|T \cap N_1| \cdot |T\cap N_2|)^2 \mathbb{P}(T \text{ is a component in } G)\\
	 &= o\left( \log^{9} n_1 \right)p^2\sum_{T \in \mathcal{T}} \mathbb{P}(T \text{ is a component in } G)\\
	  &= o\left( \log^{9} n_1 \right)p^2\mathbb{E}(X_T)\\
	 &\leq o\left( \log^{9} n_1 \right) p^2(n_1+n_2)\\
	 &= o(1),
	\end{align*}
	since $X_T \leq n_1 + n_2$. It then follows by Markov's inequality that whp $X_C = 0$.
\end{proof}
\end{document}